\documentclass[11pt]{amsart}

\usepackage[letterpaper,top=1.1in,bottom=1.1in,left=0.9in,right=0.9in]{geometry}
\usepackage{amsmath,amssymb,mathtools}
\usepackage{microtype}
\usepackage{xcolor}
\definecolor{linkviolet}{RGB}{112,60,160}
\definecolor{citeblue}{RGB}{35,82,145}
\colorlet{leanblue}{blue!55!black}
\usepackage{hyperref}
\hypersetup{
  colorlinks=true,
  linkcolor=linkviolet,
  citecolor=citeblue,
  urlcolor=citeblue,
  pdftitle={The Parafree Conjecture for Lie rings and associative algebras},
  pdfauthor={Vasily Ionin and Roman Mikhailov},
}
\usepackage{quiver}

\usepackage{enumitem}

\numberwithin{equation}{section}
\newtheorem{theorem}{Theorem}

\newtheorem*{problems*}{Problems}

\newtheorem{proposition}[equation]{Proposition}
\newtheorem{lemma}[equation]{Lemma}
\newtheorem{corollary}[equation]{Corollary}
\theoremstyle{definition}
\newtheorem{remark}[equation]{Remark}

\newcommand{\leanrepository}{%
  https://github.com/madoka-wizard/parafree-lie-associative/blob/main}
\newcommand{\leanproof}[1]{%
  \leavevmode\unskip\nobreak\hfill
  \mbox{\href{#1}{\normalfont\sffamily\bfseries\footnotesize
    \textcolor{leanblue}{Lean 4}}}%
  \par\nobreak\noindent\ignorespaces
}

\newcommand{\Z}{\mathbb{Z}}
\newcommand{\Q}{\mathbb{Q}}
\newcommand{\F}{\mathbb{F}}
\newcommand{\kk}{\Bbbk}
\newcommand{\FL}{\mathrm{FL}}
\newcommand{\Zf}{\Z_{(5)}}
\newcommand{\ab}{\mathrm{ab}}
\newcommand{\ad}{\operatorname{ad}}
\newcommand{\presentationlangle}{%
  \mathopen{\vcenter{\hbox{%
    \begin{tikzpicture}[x=1ex,y=1ex,baseline=(current bounding box.center)]
      \draw[line width=.45pt] (2,4.5) -- (0,0) -- (2,-4.5);
    \end{tikzpicture}%
  }}}%
}
\newcommand{\presentationrangle}{%
  \mathclose{\vcenter{\hbox{%
    \begin{tikzpicture}[x=1ex,y=1ex,baseline=(current bounding box.center)]
      \draw[line width=.45pt] (0,4.5) -- (2,0) -- (0,-4.5);
    \end{tikzpicture}%
  }}}%
}
\title[The Parafree Conjecture for Lie rings and associative algebras]{The Parafree Conjecture \\ for Lie rings and associative algebras}
\author{Vasily Ionin}
\address{Saint Petersburg Department of Steklov Mathematical Institute\newline
27 Fontanka, St.~Petersburg, 191023, Russia}
\email{ionin.code@gmail.com}
\author{Roman Mikhailov}
\address{Saint Petersburg State University\newline
7--9 Universitetskaya Embankment, St.~Petersburg, 199034, Russia}
\date{}

\begin{document}

\begin{abstract}
\leftskip=1.6pc
\rightskip=1.6pc
\hyphenpenalty=10000
We construct a finitely generated parafree Lie ring with nonzero second
homology and, over every nonzero commutative ring, a finitely generated
parafree augmented associative algebra with nonzero second homology. These
examples disprove the analogues of the Parafree Conjecture for Lie rings and
for augmented associative algebras. Our constructions also provide finitely
generated parafree examples that are not finitely presented, both in the
Lie-ring and augmented-algebra settings.
\end{abstract}

\maketitle

\section{Introduction}

A group $G$ is called
\emph{parafree} if it is residually nilpotent and there is a free group $F$
and a homomorphism $F\to G$ inducing
isomorphisms
\(F/\gamma_n(F) \cong G/\gamma_n(G)\)
for all $n \geq 1$.
Here, $\gamma_1(G)=G$ and $\gamma_{n+1}(G)=[\gamma_n(G),G]$ denote
the lower central series of a group $G$.

Parafree groups were introduced by G.~Baumslag in the 1960s in connection
with the problem of characterizing groups of cohomological dimension one
\cite{Baumslag1967a,Baumslag1967b,Baumslag1968,Baumslag1969}. Baumslag
hoped that nonfree parafree groups might provide examples of nonfree groups
of cohomological dimension one. The Stallings--Swan theorem subsequently
established that every group of cohomological dimension one is free
\cite{Stallings1968,Swan1969}. Nevertheless, parafree groups proved to be of
independent interest.
They arise in questions about lower central series, fundamental groups of
$3$-manifolds, four-dimensional surgery, and link concordance; see
\cite{CochranOrr1998} and \cite[Problems~3.78 and~3.79]{Kir}.

Baumslag formulated a number of problems concerning the extent to which
properties of free groups persist for parafree groups; see, for example,
\cite{Baumslag2005}.
\begin{problems*}
For a finitely generated parafree group $G$, is it true that
\begin{enumerate}[label=(\alph*)]
\item it is finitely presented?
\item it is hyperbolic?
\item it is linear?
\item $H_2(G) = 0$?
\item it has cohomological dimension at most $2$?
\end{enumerate}
\end{problems*}
To the best of our
knowledge, these questions remain open. Problem~(d) is commonly known as the \emph{Parafree Conjecture}. The
finite-generation assumption is essential:
non-finitely generated parafree groups
with nontrivial second homology are known to exist
\cite{Bousfield1977,IvanovMikhailovZaikovskii2020}.
Problem~(e) remains open for infinitely generated groups as well. In particular,
the cohomological dimension of the pronilpotent completion $\widehat{F}$ of a
finite-rank free group $F$ is unknown.

The notion of parafreeness extends naturally to other algebraic categories,
including Lie algebras and augmented associative algebras. This leads to
analogous questions: which properties are shared by free and parafree objects,
and to what extent do finiteness properties and homological dimension depend
only on their nilpotent quotients? Parafree Lie algebras were introduced in \cite{BS80}. A countable parafree Lie algebra over $\F_2$ with nonzero $H_2$ was constructed in
\cite{IvanovMikhailovZaikovskii2020}. For augmented associative algebras, a systematic theory was developed by Ivanov and Lopatkin \cite{IL21}. In
particular, they constructed a finitely generated parafree augmented
associative algebra of infinite cohomological dimension. They also
asked whether
\(
  H_2(B)=0
\)
for every finitely generated parafree augmented algebra $B$.

In this paper, we disprove the Lie-ring analogue of the Parafree Conjecture
for finitely generated Lie rings and the associative-algebra analogue for
finitely generated augmented algebras over every nonzero commutative base
ring.
We also answer Problem~(a) in the negative for Lie rings and for associative algebras.

\subsection{The Parafree Conjecture for Lie rings}
A \emph{Lie ring} is a Lie algebra over $\mathbb{Z}$. We use right-normed brackets:
\([z_1,\ldots,z_n]=[z_1,[z_2,\ldots,z_n]]\).
For a Lie ring~$L$, let $\gamma_1(L)=L$ and
$\gamma_{n+1}(L)=[L,\gamma_n(L)]$ for $n\geq1$.
A Lie ring $L$ is \emph{residually nilpotent} if
\(\gamma_\omega(L):= \bigcap_{n\geq1}\gamma_n(L)=0\).

Let $\mathrm{FL}_{\Bbbk}(X)$ denote the free Lie algebra on a set $X$.
We call $L$ \emph{almost parafree of rank $m$} if, for some set $X$ of $m$
elements, there is a homomorphism $F=\mathrm{FL}_{\mathbb Z}(X)\to L$ inducing
isomorphisms
\[
  F/\gamma_n(F)\cong L/\gamma_n(L)\qquad\text{for all } n\geq1.
\]
If, in addition, $L$ is residually nilpotent, then $L$ is \emph{parafree}.

We write $H_i(L) = H_i(\mathsf{CE}_\bullet(L))$, where $\mathsf{CE}_\bullet(L)$ is the Chevalley--Eilenberg complex.

Our main results for Lie rings are the following.
\begin{theorem}\label{thm:A}
\leanproof{\leanrepository/Parafree/MainResults.lean\#L26-L39}
Let \(L_0\) be a Lie ring given by the presentation
\[
L_0=\left\langle q_1,q_2,q_3, z,u,v
\;\middle|\;
z=[u,v]+2q_1,\; u=[v,z]+2q_2,\; v=2q_3-[z,u]
\right\rangle,
\]
and let $L=L_0/\gamma_\omega(L_0)$. Then:
\begin{enumerate}[label=\textup{(\arabic*)}]
\item $L$ is parafree of rank $3$: it is residually nilpotent, and the
morphism $F=\FL_\Z(q_1,q_2,q_3)\to L$, $q_i\mapsto q_i$, induces
isomorphisms $F/\gamma_n(F)\cong L/\gamma_n(L)$ for all $n\geq1$;
\item $H_2(L)\otimes\mathbb{F}_2\cong\mathbb{F}_2$.
\end{enumerate}
\end{theorem}

In the proof, $L$ is identified with the Lie subring of $\FL_\Q(z,u,v)$ generated by $z,u,v$ and $\frac12(z-[u,v])$, $\frac12(u-[v,z])$, $\frac12(v+[z,u])$.

\begin{theorem}\label{thm:B}
\leanproof{\leanrepository/Parafree/MainResults.lean\#L45-L60}
In $\FL_\Q(t,a)$ put $w=[t,a]-2a$ and
\[
f_1=[t,t,w]-w,\qquad f_2=[a,t,t,a]-[w,w,t,w].
\]
Let $P = \langle t, a, f_1/5, f_2/5 \rangle_\mathsf{LieRing} \subseteq \FL_\Q(t, a)$.
Then:
\begin{enumerate}
\item $P$ is parafree of rank $2$: it is residually nilpotent, and the
morphism $F=\FL_\Z(\xi_1,\xi_2)\to P$, $\xi_1\mapsto t$,
$\xi_2\mapsto a-2f_1/5$, induces isomorphisms
$F/\gamma_n(F)\cong P/\gamma_n(P)$ for all $n\geq1$;
\item
\(H_2(P)\otimes\F_5\cong\F_5^{\oplus\aleph_0}\).
\end{enumerate}
The Lie ring $P$ is finitely generated but not finitely presented.
\end{theorem}
Thus the Lie-ring analogue of the Parafree Conjecture fails, as
does the analogue of Problem~(a).

In both theorems, we use the same strategy to prove parafreeness:
embedding into a free Lie algebra over $\mathbb Q$ implies residual
nilpotence, while Lemma~\ref{lem:lie-implicit} gives a simple presentation
criterion for almost parafreeness. To compute $H_2$, we use the following
criterion from Proposition~\ref{prop:lie-divided-relators}. Let
$f_1,\ldots,f_h\in\mathrm{FL}_{\mathbb Z}(z_1,\ldots,z_d)$, and suppose
that
\[
 G=\mathrm{FL}_{\mathbb Z_{(p)}}(z_1,\ldots,z_d)/(f_1,\ldots,f_h)
\]
is a free $\mathbb Z_{(p)}$-module. Let
$D\subseteq\mathrm{FL}_{\mathbb Q}(z_1,\ldots,z_d)$ be the Lie ring
generated by the $z_i$ and the \emph{divided relators} $f_i/p$. Then
\[
 H_2(D)\otimes\mathbb F_p\cong H_3(G\otimes\mathbb F_p)
\]
provided that $H_2(G\otimes\mathbb F_p)=0$ and $D_{\mathrm{ab}}$ has no
$p$-torsion.

\subsection{The Parafree Conjecture for associative algebras}

Fix a nonzero commutative ring $\Bbbk$. All algebras are
associative unital $\Bbbk$-algebras, and all ideals
are two-sided. An \emph{augmentation} of an algebra $B$ is a unital
$\Bbbk$-algebra homomorphism $\varepsilon_B\colon B\to\Bbbk$; its kernel
$I(B)=\ker(\varepsilon_B)$ is the \emph{augmentation ideal}.
An augmented algebra $B$ is \emph{residually nilpotent} if
\(I(B)^\omega:=\bigcap_{n\geq1}I(B)^n=0\).

The free algebra $\Bbbk\langle X\rangle$ on a set $X$ is augmented by
$x\mapsto0$. An augmented algebra $B$ is \emph{almost parafree of rank $m$}
if, for some set $X$ of $m$ elements, there is a morphism of augmented
algebras $F=\Bbbk\langle X\rangle\to B$ inducing isomorphisms
\[
  F/I(F)^n\cong B/I(B)^n\qquad\text{for all } n\geq1.
\]
If, in addition, $B$ is residually nilpotent, then $B$ is \emph{parafree}.

We write
\(
  H_i(B)=\operatorname{Tor}^B_i(\Bbbk,\Bbbk),
\)
where the copies of $\Bbbk$ are regarded as $B$-modules via the augmentation.

Our main result for associative algebras is the following.
\begin{theorem}\label{thm:main}
\leanproof{\leanrepository/Parafree/MainResults.lean\#L66-L79}
For $\ell\geq0$, set
\begin{equation}\label{eq:def-e}
 e_\ell=(1+a)(1+b)^\ell(1+c)-(1+d)(1+x)^\ell(1+y).
\end{equation}
Let $A$ be an augmented algebra given by the presentation
\begin{equation}\label{eq:def-A}
 A=\Bbbk\langle a,b,c,d,x,y\rangle/(e_\ell\mid\ell\geq0).
\end{equation}
Then:
\begin{enumerate}
\item $A$ is parafree of rank $4$: it is residually nilpotent, and the
morphism $F=\Bbbk\langle a,b,c,d\rangle\to A$, $(a,b,c,d)\mapsto(a,b,c,d)$,
induces isomorphisms $F/I(F)^n\cong A/I(A)^n$ for all $n\geq1$;
\item $H_2(A)\cong\Bbbk^{\oplus\aleph_0}$ as a $\Bbbk$-module.
\end{enumerate}
The algebra $A$ is finitely generated but not finitely presented.
\end{theorem}

Thus the associative-algebra analogue of the Parafree Conjecture fails, as
does the analogue of Problem~(a).

\begin{remark}
The algebra $A$ is the monoid algebra $\Bbbk[L]$ of a finitely generated but not
finitely presented submonoid $L$ of a free monoid; see
Corollary~\ref{cor:monoid}. Such monoids were first constructed by
Spehner~\cite[Theorem~3.2]{Spehner1977}.
\end{remark}

To show that $A$ is parafree, we introduce the auxiliary almost parafree algebra
\begin{equation}\label{eq:def-A0}
 \begin{aligned}
 A_0&=\Bbbk\langle a,b,c,d,x,y\rangle/(e_0,e_1)\\
 &=\Bbbk\left\langle
 a,b,c,d,x,y
 \;\middle|\;
 \begin{aligned}
 x&=b+ab+bc-dx-xy+abc-dxy,\\
 y&=a+c-d+ac-dy
 \end{aligned}
 \right\rangle,
 \end{aligned}
\end{equation}
and show that $A = A_0 / I(A_0)^\omega$; see Proposition~\ref{prop:A_is_parafree}.
In Proposition~\ref{prop:identification_of_H_2_A}, we
identify $H_2(A)$ with the quotient
\(I(A_0)^\omega/\left(I(A_0)^\omega\cdot I(A_0)+I(A_0)\cdot I(A_0)^\omega\right)\).

For an ideal $I$ of an algebra $B$, define its \emph{transfinite powers}
$I^\alpha$, indexed by ordinals, recursively by
\[
 I^0=B,
 \qquad
 I^{\alpha+1}=I^\alpha\cdot I+I\cdot I^\alpha,
 \qquad
 I^\alpha=\bigcap_{\beta<\alpha}I^\beta
 \quad\text{for limit }\alpha.
\]
The \emph{transfinite length} of $\{I^\alpha\}$ is the least ordinal $\alpha$
such that $I^\alpha=I^{\alpha+1}$. Proposition~\ref{prop:identification_of_H_2_A} and
Theorem~\ref{thm:main} imply that
\(I(A_0)^\omega\neq I(A_0)^{\omega+1}\).
\begin{theorem}\label{thm:transfinite-length}
\leanproof{\leanrepository/Parafree/MainResults.lean\#L85-L98}
The series $\{I(A_0)^\alpha\}$ has transfinite length $\omega^2$.
\end{theorem}
\subsection{Comparing Lie and associative cases}

Both constructions give subalgebras of free objects, but they differ in two
ways. First, the associative example lies in a free algebra over the same
ground ring, whereas our realization of the Lie ring as a subring of a free
Lie algebra uses the extension $\mathbb Z\subseteq\mathbb Q$.

The associative-algebra construction and its proof work over every nonzero
commutative base ring, including $\mathbb Z$. For the Lie example, however,
working over a ring rather than a field is essential: every Lie subalgebra of
a free Lie algebra over a field is free, by the Shirshov--Witt theorem.

\subsection{Organization}

The Lie-ring part
(Sections~\ref{sec:lie-preliminaries}--\ref{sec:proof-thm-B}) and the
associative part (Sections~\ref{sec:preliminaries}--\ref{sec:homology}) are
independent of one another and may be read separately.

Section~\ref{sec:lie-preliminaries} recalls the homological tools for Lie
rings, and Section~\ref{sec:lie-divided-relators} establishes the
divided-relator result. Section~\ref{sec:proof-thm-A} applies these tools to
prove Theorem~\ref{thm:A}, and Section~\ref{sec:proof-thm-B} proves
Theorem~\ref{thm:B}.

Section~\ref{sec:preliminaries} recalls completions, monoid algebras,
$2$-acyclicity, and a finiteness criterion for second homology. In
Section~\ref{sec:monoid}, we identify $A$ with the monoid algebra
$\Bbbk[L]$ of a finitely generated but not finitely presented submonoid $L$
of a free monoid. In
Section~\ref{sec:completion}, we prove that $A$ is parafree of rank four. In
Section~\ref{sec:homology}, we compute $H_2(A)$ and determine the
transfinite augmentation series of $A_0$.

In Section~\ref{sec:formalization}, we describe the Lean~4 formalization
of the four main theorems and explain how to access and check the proofs.

\subsection{Acknowledgements}

The authors are deeply grateful to Laurent Bartholdi and Artem Semidetnov for
stimulating discussions and their infectious enthusiasm for this circle of
problems.

\subsection{AI use statement}

This text was written with the assistance of AI systems.
Working with AI
was an iterative process, making it difficult to isolate its specific
contribution.

\section{Preliminaries on Lie algebras}\label{sec:lie-preliminaries}

All Lie algebras have an alternating bracket. For a Lie algebra
$L$ over a commutative ring $\Bbbk$, write
\[
  L_{\mathrm{ab}}=L/[L,L],
  \qquad
  \operatorname{gr}_\gamma L
  =\bigoplus_{n\geq1}\gamma_n(L)/\gamma_{n+1}(L).
\]
For a $\Bbbk$-module $M$, we write $\mathrm{FL}_{\Bbbk}(M)$ for
the free Lie algebra generated by $M$. If $M$ is free with basis $X$, then we also write $\mathrm{FL}_{\Bbbk}(X) \vcentcolon= \mathrm{FL}_{\Bbbk}(M)$. The Hall basis shows that each
homogeneous component of $\mathrm{FL}_{\Bbbk}(X)$ is free over $\Bbbk$.

In Sections~\ref{sec:lie-preliminaries}--\ref{sec:proof-thm-B}, $H_i(L;M)$ denotes Chevalley--Eilenberg homology
with coefficients in a $L$-module $M$, and $H_i(L)$
denotes homology with trivial coefficients in the ground ring. The latter
is computed by $\mathsf{CE}_n(L)=\Lambda^n_{\Bbbk}L$,
with differential
\[
  d(x_1\wedge\cdots\wedge x_n)
  =\sum_{i<j}(-1)^{i+j+1}[x_i,x_j]\wedge
  x_1\wedge\cdots\wedge\widehat{x_i}\wedge\cdots
  \wedge\widehat{x_j}\wedge\cdots\wedge x_n.
\]
In particular, $H_1(L)=L_{\mathrm{ab}}$.
If $L$ is free as a module over $\Bbbk$, these homology groups agree with
$\operatorname{Tor}^{U(L)}_i(\Bbbk,\Bbbk)$
\cite[Theorem~7.3]{IPRZ21}.
For a free Lie algebra
$F=\mathrm{FL}_{\Bbbk}(X)$, the resolution
\begin{equation}\label{eq:lie-free-resolution}
  0\to U(F)^{\oplus |X|}
  \xrightarrow{(b_i)\mapsto\sum_i b_i z_i}U(F)
  \to\Bbbk\to0
\end{equation}
shows that $H_i(F)=0$ for $i\geq2$.

For any extension $0\to N\to L\to P\to0$
of Lie algebras over a principal ideal domain, there is the five-term exact sequence~\cite{HS53}
\begin{equation}\label{eq:lie-five-term}
  H_2(L)\to H_2(P)
  \to\frac{N}{[L,N]}
  \to L_{\mathrm{ab}}
  \to P_{\mathrm{ab}}\to0.
\end{equation}

\begin{lemma}\label{lem:lie-implicit}
Let $z=(z_1,\ldots,z_d)$ and $y=(y_1,\ldots,y_h)$. If each
$h_i(z,y)$ belongs to the derived ideal of $\mathrm{FL}_{\mathbb Z}(z,y)$,
then the natural map
\[
  \mathrm{FL}_{\mathbb Z}(z)\to
  \langle z,y\mid y_i=h_i(z,y),\ 1\leq i\leq h\rangle_{\mathsf{LieRing}}
\]
induces isomorphisms on all lower central quotients.
\end{lemma}

\begin{proof}
Let $F=\mathrm{FL}_{\mathbb Z}(z,y)$, let $r_i=y_i-h_i(z,y)$, and let
$R$ be the ideal generated by the $r_i$, so the Lie ring on the right is
$L=F/R$. Since $R$ is generated as an ideal by the $r_i$, modulo $[F,R]$
every element of $R$ is an integer linear combination of the $r_i$; thus
$R=\sum_i\mathbb Zr_i+[F,R]$. Take $x\in R\cap[F,F]$ and write
$x=\sum_i n_i r_i+u$, where $n_i\in\mathbb Z$ and $u\in[F,R]$.
Passing to $F_{\mathrm{ab}}$ gives $0=\sum_i n_i\overline{r_i}$. Since each $h_i$ lies in $[F,F]$,
$\overline{r_i}=\overline{y_i}$; these classes are part of a free basis
of $F_{\mathrm{ab}}$, so every $n_i=0$. Hence $x=u\in[F,R]$.
Conversely, $[F,R]\subseteq R\cap[F,F]$, proving the equality. The Hopf
formula therefore gives $H_2(L)=0$
\cite[Theorem~8.4]{IPRZ21}.

Also $L_{\mathrm{ab}}$ is freely generated by
the $z_i$, so the natural map $\mathrm{FL}_{\mathbb Z}(z)\to L$ is an
isomorphism on $H_1$ and is surjective on $H_2$, since both groups
vanish. Thus~\cite[Theorem~2.14]{IvanovMikhailovZaikovskii2020} implies that it
induces isomorphisms on all lower central quotients.
\end{proof}

\subsection{Finitely presented Lie rings}

\begin{lemma}\label{lem:lie-fp-homology}
If $L$ is a finitely presented Lie ring, then $H_2(L)$ is a finitely
generated abelian group.
\end{lemma}

\begin{proof}
Write $L=F/R$, where $F$ is a free Lie ring of finite rank and $R$ is
generated as an ideal by finitely many relators $r_1,\ldots,r_s$. By the
Hopf formula~\cite[Theorem~8.4]{IPRZ21},
\[
H_2(L)\cong\frac{R\cap[F,F]}{[F,R]}
\hookrightarrow\frac{R}{[F,R]}.
\]
The images of $r_1,\ldots,r_s$ generate $R/[F,R]$ as an abelian group,
since every iterated bracket of an element of $F$ with an element of $R$
lies in $[F,R]$. Thus $R/[F,R]$ is finitely generated, and its subgroup
$H_2(L)$ is finitely generated as well.
\end{proof}

\section{Homology of Lie rings given by divided relators}
\label{sec:lie-divided-relators}

\begin{lemma}\label{lem:relation_sequence}
Let $\Bbbk$ be a principal ideal domain, $V$ be a free $\Bbbk$-module, and $R$ be an ideal in $F = \FL_\Bbbk(V)$ such that $G = F/R$ is a free $\Bbbk$-module. Then $R_\ab$ is a free $\Bbbk$-module and there is an exact sequence of left $U(G)$-modules
\[0 \to R_\ab \to U(G) \otimes V \to U(G) \to \Bbbk \to 0.\]
\end{lemma}

\begin{proof}
Since $R$ and $G$ are free $\Bbbk$-modules, a basis of $R$ extends to a
basis of $F$. The Poincar\'e--Birkhoff--Witt theorem makes $U(F)$ free
as a left $U(R)$-module. Thus the free resolution~\eqref{eq:lie-free-resolution}
of $\Bbbk$ over $U(F)$ is a free $U(R)$-resolution. Applying
$\Bbbk\otimes_{U(R)}$ gives a complex $U(G)\otimes V\to U(G)$ computing
$H_*(R)$; its homology in degrees zero and one is $\Bbbk$ and $R_\ab$,
respectively, and it vanishes in higher degrees. This yields the exact
sequence, and $R_\ab$, as a submodule of the free $\Bbbk$-module
$U(G)\otimes V$, is free.
\end{proof}

\begin{lemma}\label{lem:gr_R_decomposition}
Let $\Bbbk$ be a principal ideal domain, $V$ be a free $\Bbbk$-module, and $R$ be an ideal in $F = \FL_\Bbbk(V)$ such that $G = F/R$ is a free $\Bbbk$-module. Then there is an isomorphism of graded Lie algebras
\[\bigoplus_{i \ge 1} \gamma_i(R) / \gamma_{i+1}(R) \overset{\sim}{\to} \FL_\Bbbk(R_\mathrm{ab}).\]
\end{lemma}

\begin{proof}
Put $Q=\FL_\Bbbk(R_\ab)$ and choose lifts to $R$ of a basis of $R_\ab$.
They induce a homomorphism $\phi\colon Q\to R$ that is an isomorphism on
abelianizations, hence modulo $\gamma_2$. Suppose inductively that it induces
an isomorphism modulo $\gamma_n$, for some $n\ge2$. The Hall basis shows
that $Q/\gamma_n(Q)$ is free, so $R/\gamma_n(R)$ is free as well. For
$X=Q,R$, we have $H_2(X)=0$: this follows from the free-Lie resolution
for $Q$ and from Lemma~\ref{lem:relation_sequence} for $R$. The
five-term sequence~\eqref{eq:lie-five-term} for
$0\to\gamma_n(X)\to X\to X/\gamma_n(X)\to0$ therefore gives natural
isomorphisms
\[
 H_2(X/\gamma_n(X))\cong
 \gamma_n(X)/\gamma_{n+1}(X),
\]
since $X_\ab\to(X/\gamma_n(X))_\ab$ is an isomorphism. Naturality and the
induction hypothesis show that $\phi$ induces an isomorphism on these
$n$th lower-central quotients. The five lemma applied to the lower-central
quotient sequences then gives an isomorphism modulo $\gamma_{n+1}$. Induction
proves that $\phi$ induces an isomorphism on every graded component, as
required.
\end{proof}

\begin{proposition}\label{prop:lie-divided-relators}
Let $p$ be a prime and let
$f_1,\ldots,f_h\in\mathrm{FL}_{\mathbb Z}(z_1,\ldots,z_d)$.
Let $R$ be the ideal in
$F=\mathrm{FL}_{\mathbb Z_{(p)}}(z_1,\ldots,z_d)$ generated by $f_1, \ldots, f_h$,
and assume that $G=F/R$ is free as a
$\mathbb Z_{(p)}$-module. Put
\[
   D=
  \langle z_1,\ldots,z_d,f_1/p,\ldots,f_h/p\rangle_\mathsf{LieRing}
  \subseteq\mathrm{FL}_{\mathbb Q}(z_1,\ldots,z_d).
\]
\begin{enumerate}
\item The additive group of $ D$ is free, and
\[D\otimes\mathbb F_p\cong
  \mathrm{FL}_{\mathbb F_p}(R_{\mathrm{ab}}\otimes\mathbb F_p)
  \rtimes (G \otimes\mathbb F_p),\]
where the action extends the adjoint action of $G \otimes \mathbb{F}_p$ on $R_{\mathrm{ab}}\otimes\mathbb F_p$.
\item If the abelian group $D_{\mathrm{ab}}$ has no $p$-torsion and
$H_2(G\otimes\mathbb F_p)=0$, then
\(H_2(D)\otimes\mathbb F_p \cong H_3(G\otimes\mathbb F_p)\).
\end{enumerate}
\end{proposition}

\begin{proof}
(1) The map $z_i\mapsto pz_i$ sends each $f_i/p$ to an integral Lie
polynomial. It therefore embeds $D$ into
$\mathrm{FL}_{\mathbb Z}(z_1,\ldots,z_d)$, so the additive group of $D$ is free.

Let $D_{(p)}$ be the $\mathbb{Z}_{(p)}$-span of $D$ in
$\FL_\Q(z_1,\ldots,z_d)$. Then
$D_{(p)}\cong D\otimes\mathbb{Z}_{(p)}$ and
$D_{(p)}/pD_{(p)}\cong D\otimes\mathbb{F}_p$. Define a descending
filtration of $F$ by $E_0=F$ and $E_i=\gamma_i(R)$ for $i\ge1$.
It satisfies $[E_i,E_j]\subseteq E_{i+j}$ for all $i,j\ge0$.
We claim that there is an equality
\begin{equation}\label{eq:D_p_decomposition}
D_{(p)} = \sum_{j \ge 0} p^{-j} E_j
\end{equation}
of $\mathbb{Z}_{(p)}$-submodules of $\FL_\Q(z_1, \ldots, z_d)$.
The right-hand side is a Lie subring containing each $z_i\in E_0$ and
$f_i/p\in p^{-1}E_1$, so it contains $D_{(p)}$. Conversely, $E_0=F$
is contained in $D_{(p)}$. The ideal $R=E_1$ is spanned over
$\mathbb{Z}_{(p)}$ by brackets $[x_1,\ldots,x_k,f_i]$ with $x_j\in F$,
so $p^{-1}E_1\subseteq D_{(p)}$. For $i\ge2$, $E_i=\gamma_i(R)$ is
spanned by brackets $[r_1,\ldots,r_i]$ with $r_j\in R$, and hence
\[
p^{-i}[r_1,\ldots,r_i]=[r_1/p,\ldots,r_i/p]\in D_{(p)}.
\]
Thus $p^{-i}E_i\subseteq D_{(p)}$ for every $i$, proving the equality.

Consider the homomorphism of $\mathbb{Z}_{(p)}$-modules
\[\Theta \colon \bigoplus_{i \ge 0} E_i/(E_{i+1}+pE_i)
\to D_{(p)}/pD_{(p)}\cong D\otimes\mathbb{F}_p,
\qquad [x]\longmapsto[p^{-i}x],\quad x\in E_i.\]
This is well-defined: $p^{-i}E_{i+1}\subseteq pD_{(p)}$ by the
decomposition above, and $p^{-i}(pE_i)\subseteq pD_{(p)}$ as well. The
source is a graded Lie algebra with bracket induced by
$[E_i,E_j]\subseteq E_{i+j}$, and $\Theta$ preserves brackets.
We claim that $\Theta$ is an isomorphism of Lie rings. It is surjective by~\eqref{eq:D_p_decomposition}. Consider the ascending filtration
\[S_0\subseteq S_1\subseteq\cdots\subseteq D_{(p)},
\qquad S_m=\sum_{i\le m}p^{-i}E_i.\]
Take $\sum_{i\le N_0}[x_i]\in\ker(\Theta)$, with $x_i\in E_i$.
Then
$\sum_{i} p^{-i} x_i = py$ for some $y \in D_{(p)}$.
Choose $N \ge N_0$ such that $y \in S_N$.
Set $x_i=0$ for $N_0<i\le N$.
The quotients $E_i/E_{i+1}$ are free: for $i=0$ this is $G$, and for
$i\ge1$ this follows from Lemma~\ref{lem:gr_R_decomposition} and the Hall
basis. Since $E_N\subseteq F$ is free, choose complements with bases
$\mathcal B_i$ for $i<N$ and a basis
$\mathcal B_N$ of $E_N$, so $E_i=\langle\mathcal B_i\rangle\oplus E_{i+1}$
for $i<N$. Then $S_N$ is free on $\{p^{-i}b:i<N,\ b\in\mathcal B_i\}$
and $\{p^{-N}b:b\in\mathcal B_N\}$. For $i<N$ write
$x_i=x_i'+x_i''$ with $x_i'\in\langle\mathcal B_i\rangle$ and
$x_i''\in E_{i+1}$, and put $x_N'=x_N$. The residual terms
$p^{-i}x_i''$ lie in $pS_N$, so the kernel relation implies
$\sum_i p^{-i}x_i'\in pS_N$. Linear independence of the chosen basis
shows that all coefficients of the $x_i'$ are divisible by $p$. Thus
$x_i\in E_{i+1}+pE_i$ for $i<N$, while $x_N\in pE_N$. Every $[x_i]$
is therefore zero, proving injectivity.

By Lemma~\ref{lem:gr_R_decomposition}, this graded Lie algebra is
\[
F/(R+pF)\oplus\bigoplus_{i\ge1}
\gamma_i(R)/(\gamma_{i+1}(R)+p\gamma_i(R))
\cong (G\otimes\mathbb F_p)\oplus
\FL_{\mathbb F_p}(R_{\mathrm{ab}}\otimes\mathbb F_p).
\]
The first summand is a Lie subalgebra, the second an ideal, and the
induced action is the adjoint action of $G\otimes\mathbb F_p$.

(2) Consider the Hochschild--Serre spectral sequence \cite{HS53} for the split extension constructed in~(1):
\[E^2_{st} = H_s(G \otimes \mathbb{F}_p; H_t(\FL_{\mathbb{F}_p}(R_\ab \otimes \mathbb{F}_p))) \Rightarrow H_{s+t}(D \otimes \mathbb{F}_p).\]
Since $\FL_{\mathbb{F}_p}(R_\ab \otimes \mathbb{F}_p)$ is free, $E^2_{st}$ vanishes for $t \ge 2$. The projection $D \otimes \mathbb{F}_p \twoheadrightarrow G \otimes \mathbb{F}_p$ has a Lie algebra section, so the edge maps
\[H_s(D \otimes \mathbb{F}_p) \to H_s(G \otimes \mathbb{F}_p) = E^2_{s,0}\]
are surjective. Hence every bottom-row class survives, so all differentials from that row vanish. Differentials from row $t=1$ have zero target, and all rows $t\ge2$ vanish. Since $H_2(G \otimes \mathbb{F}_p) = 0$ by assumption,
\[H_2(D \otimes \mathbb{F}_p) \cong H_1(G \otimes \mathbb{F}_p; H_1(\FL_{\mathbb{F}_p}(R_\ab \otimes \mathbb{F}_p))) = H_1(G \otimes \mathbb{F}_p; R_\ab \otimes \mathbb{F}_p).\]

By Lemma~\ref{lem:relation_sequence}, there is an exact sequence
\[0 \to R_\ab \to U(G)^{\oplus d} \to U(G) \to \Z_{(p)} \to 0\]
of left $U(G)$-modules. This remains exact after tensoring over
$\mathbb{Z}_{(p)}$ with $\mathbb{F}_p$: the image of the map
$U(G)^{\oplus d}\to U(G)$ is the augmentation ideal, a submodule of the
free $\mathbb{Z}_{(p)}$-module $U(G)$, and is therefore free. The quotient
in the second short exact sequence is $\mathbb{Z}_{(p)}$, also free. Thus
both sequences split over $\mathbb{Z}_{(p)}$. We obtain
\[0 \to R_\ab \otimes \mathbb{F}_p \to U(G \otimes \mathbb{F}_p)^{\oplus d} \to U(G \otimes \mathbb{F}_p) \to \mathbb{F}_p \to 0\]
of left $U(G \otimes \mathbb{F}_p)$-modules.
Dimension shifts give
\[H_1(G \otimes \mathbb{F}_p; R_\ab \otimes \mathbb{F}_p) = H_3(G \otimes \mathbb{F}_p).\]
Finally, $\Lambda^\bullet_\Z D$ is a complex of free abelian groups with $H_1=D_\ab$. Since $D_\ab$ has no $p$-torsion, the universal coefficient theorem gives $H_2(D)\otimes\F_p\cong H_2(D\otimes\F_p)$.
\end{proof}

\section{Proof of Theorem~\ref{thm:A}}\label{sec:proof-thm-A}
Let \(L_0\) be a Lie ring given by the presentation
\[
L_0=\left\langle q_1,q_2,q_3, z,u,v
\;\middle|\;
z=[u,v]+2q_1,\; u=[v,z]+2q_2,\; v=2q_3-[z,u]
\right\rangle.
\]
\begin{proposition}\label{prop:thm-A-lcq}
The natural map $\FL_\Z(q_1,q_2,q_3)\to L_0$ induces isomorphisms on all
lower central quotients.
\end{proposition}

\begin{proof}
Set $y_1=z-2q_1$, $y_2=u-2q_2$, and $y_3=v-2q_3$ in the presentation
of $L_0$. This integral change of generators rewrites its relations as
\[
\begin{aligned}
y_1&=[y_2+2q_2,y_3+2q_3],\\
y_2&=[y_3+2q_3,y_1+2q_1],\\
y_3&=-[y_1+2q_1,y_2+2q_2].
\end{aligned}
\]
Each right-hand side belongs to the derived ideal. The result follows from
Lemma~\ref{lem:lie-implicit}.
\end{proof}

Set $L = L_0/\gamma_\omega(L_0)$. In $\FL_\Z(z,u,v)$ put
\[
g_1=z-[u,v],\qquad g_2=u-[v,z],\qquad g_3=v+[z,u].
\] Let $D$ be the Lie
subring of $\FL_\Q(z,u,v)$ generated by $z,u,v,g_1/2,g_2/2,g_3/2$.

\begin{proposition}\label{prop:thm-A-identification}
The morphism $\pi\colon L_0\to D$ sending $(z, u, v, q_i) \mapsto (z,u,v,g_i/2)$ has kernel $\gamma_\omega(L_0)$.
\end{proposition}

\begin{proof}
The defining relations of $L_0$ hold in $D$, so $\pi$ is well-defined. In
$D$ they also give
\[
D=\Z \tfrac{g_1}{2}+\Z \tfrac{g_2}{2}+\Z \tfrac{g_3}{2}+[D,D].
\]
The linear parts of $\tfrac{g_1}{2},\tfrac{g_2}{2},\tfrac{g_3}{2}$ are $z/2,u/2,v/2$, respectively, and are
linearly independent over $\Q$. It follows that the map
\[
\FL_\Z(\xi_1,\xi_2,\xi_3)\to D,
\qquad \xi_i\mapsto \tfrac{g_i}{2},
\]
induces isomorphisms on all lower central quotients: the equality above
gives surjectivity by induction, while injectivity follows because the lowest
nonzero homogeneous term of a nonzero free Lie polynomial remains nonzero
after substituting these independent linear parts.

The composition $\FL_\Z(q_1,q_2,q_3)\to L_0\xrightarrow{\pi}D$ is this
map. By Proposition~\ref{prop:thm-A-lcq}, the first map also induces
isomorphisms on all lower central quotients. Hence so does $\pi$. Since $D$ is a Lie subring of a free Lie
algebra, it is residually nilpotent. Therefore
$\ker\pi=\gamma_\omega(L_0)$.
\end{proof}

\begin{proposition}\label{prop:thm-A-homology}
The group $H_2(L)$ is a nonzero $2$-primary torsion group, and
$H_2(L)\otimes\F_2\cong\F_2$.
\end{proposition}

\begin{proof}
Apply Proposition~\ref{prop:lie-divided-relators} with $p=2$ and relators
$g_1,g_2,g_3$. The quotient
\[
G=\FL_{\Z_{(2)}}(z,u,v)/(g_1,g_2,g_3)
\]
is free over $\Z_{(2)}$ on $z,u,v$, with bracket table
$[u,v]=z$, $[v,z]=u$, and $[z,u]=-v$. Indeed, the relations reduce every
bracket to a linear combination of $z,u,v$, and this bracket table defines
a Lie ring on the free module with that basis.

Over $\F_2$, the Chevalley--Eilenberg differential
$\Lambda^2(G\otimes\F_2)\to G\otimes\F_2$ sends $u\wedge v$,
$v\wedge z$, and $z\wedge u$ to $z,u,-v$, respectively, so it is an
isomorphism. Since $d_2$ is injective, the differential
$d_3$ out of $\Lambda^3(G\otimes\F_2)$ is zero. Therefore
\[
H_2(G\otimes\F_2)=0,\qquad H_3(G\otimes\F_2)\cong\F_2.
\]
By Proposition~\ref{prop:thm-A-identification}, $L\cong D$ and
$D_{\mathrm{ab}}\cong\Z^3$, so Proposition~\ref{prop:lie-divided-relators}
gives
\[
H_2(L)\otimes\F_2\cong H_2(D)\otimes\F_2
\cong H_3(G\otimes\F_2)\cong\F_2.
\]

After inverting $2$, each $g_i/2$ belongs to
$\FL_{\Z[1/2]}(z,u,v)$, while $z,u,v$ already belong to $D$. Hence
$D\otimes\Z[1/2]=\FL_{\Z[1/2]}(z,u,v)$, whose second homology vanishes by
\eqref{eq:lie-free-resolution}. Localization commutes with
Chevalley--Eilenberg homology, so $H_2(L)\otimes\Z[1/2]=0$. Thus every
element of $H_2(L)$ is killed by a power of $2$.
\end{proof}

\begin{proof}[Proof of Theorem~\ref{thm:A}]
Set $F=\FL_\Z(q_1,q_2,q_3)$. By Proposition~\ref{prop:thm-A-lcq}, the map
$F\to L_0$ induces isomorphisms $F/\gamma_n(F)\cong L_0/\gamma_n(L_0)$ for
all $n\geq1$. Since $\gamma_\omega(L_0)\subseteq\gamma_n(L_0)$, the projection
$L_0\to L$ induces isomorphisms $L_0/\gamma_n(L_0)\cong L/\gamma_n(L)$.
Hence the composite $F\to L$, $q_i\mapsto q_i$, induces isomorphisms
$F/\gamma_n(F)\cong L/\gamma_n(L)$ for all $n\geq1$.
Proposition~\ref{prop:thm-A-identification}
identifies $L$ with the residually nilpotent Lie subring $D$, so $L$ is
parafree. Proposition~\ref{prop:thm-A-homology} gives the asserted
description of $H_2(L)$.
\end{proof}

\section{Proof of Theorem~\ref{thm:B}}\label{sec:proof-thm-B}

In this section, we construct a rank-two parafree Lie ring $P$ whose second
homology is an infinite torsion group. The construction follows the same
strategy as in Theorem~\ref{thm:A}, but the example itself is considerably
more delicate.

Regard $\mathbb Z_{(5)}[X] = \bigoplus_{n \ge 0} \Z_{(5)} X^n$ and
$\mathbb Z_{(5)}^{\oplus2}=\mathbb Z_{(5)}s\oplus\mathbb Z_{(5)}e$
as abelian Lie algebras. Define
\[
  G=(\mathbb Z_{(5)}[X]\rtimes\mathbb Z_{(5)}^{\oplus2})
  \rtimes\mathbb Z_{(5)}t
\]
by the following brackets on the \(\mathbb Z_{(5)}\)-module basis $\{t,s,e,X^0,X^1,X^2,\ldots\}$:
\begin{equation}\label{eq:lie-auxiliary-brackets}
\begin{aligned}
  &[X^i,X^j]=0,
    &i,j\geq0,\\
  &[s,e]=0,\quad [s,X^n]=X^{n+1},\quad [e,X^n]=-\tfrac32X^{n+2},
    &n\geq0,\\
  &[t,s]=s,\quad [t,e]=2e,\quad [t,X^n]=(n-1)X^n,
    &n\geq0.
\end{aligned}
\end{equation}

Set $\delta = \ad t$. In $\FL_\Z(t, a)$ put
\begin{align*}
w &= (\delta - 2) a, \\
f_1 &= (\delta^2 - 1) w, \\
f_2 &= [a, \delta^2 a] - [w, w, \delta w].
\end{align*}
We call the eigenvalues of $\ad t$ the \emph{weights}.

\begin{proposition}\label{prop:B-presentation}
The assignment $t\mapsto t$,
$a\mapsto\frac16(3s+2e+X^0)$ induces an isomorphism
\[
\FL_{\Zf}(t,a)/(f_1,f_2)\xrightarrow{\ \cong\ }G.
\]
\end{proposition}

\begin{proof}
In $\FL_{\Zf}(t,a)$ put
\begin{align*}
s &= -(\delta-2)(\delta+1)a = (-\delta^2 + \delta + 2)a, \\
e &= (\delta^2 - 1) a, \\
X^0 &= (\delta-1)(\delta-2)a = (\delta^2 - 3 \delta + 2) a.
\end{align*}
Then $\tfrac{1}{6}(3s+2e+X^0)=a$ and
\begin{equation}
\label{eq:first_part_of_relations}
(\delta-1)s=-f_1,\qquad (\delta-2)e=f_1,\qquad
(\delta+1)X^0=f_1.
\end{equation}
Thus, modulo $f_1$, $t$ acts on $s,e,X^0$ with weights $1,2,-1$,
respectively. These formulas give inverse maps
between $\FL_{\Zf}(t,a)/(f_1)$ and
$\FL_{\Zf}(s,e,X^0)\rtimes\Zf t$.

Substituting
\begin{align*}
\delta^2a &=\tfrac{1}{6}\delta^2 (3s+2e+X^0) \equiv \tfrac16(3s+8e+X^0), \\
w &= (\delta - 2) a = \tfrac{1}{6}(\delta - 2) (3s+2e+X^0) \equiv -\tfrac12(s+X^0)
\end{align*}
into $f_2$ gives
\begin{align*}
  12 f_2 &\equiv \tfrac{1}{3} [3s+2e+X^0, 3s+8e+X^0] +\tfrac{3}{2} [s+X^0, \underbrace{s+X^0, s-X^0}_{-2[s, X^0]}] \\
  &= 6[s,e]-\bigl(2[e,X^0]+3[s,s,X^0]\bigr)-3[X^0,s,X^0].
\end{align*}
The three terms have weights $3,1,-1$ modulo $f_1$. Applying
$(\delta-1)(\delta+1)$, $(\delta-3)(\delta+1)$, and
$(\delta-3)(\delta-1)$ isolates the respective terms up to units in
$\Zf$. Since $12$ is also a unit, imposing $f_2=0$ is equivalent to
imposing
\begin{equation}\label{eq:B-relations}
[s,e]=0,\qquad 2[e,X^0]+3[s,s,X^0]=0,\qquad [X^0,s,X^0]=0.
\end{equation}

Put $X^i := (\ad_s)^i X^0$ for $i\ge0$. It follows that
\[
\FL_{\Zf}(t,a)/(f_1,f_2)\cong
\presentationlangle
\left.
s,e,t,X^n\; n\ge0
\ \middle|
\begin{aligned}
&X^{n+1}=[s,X^n],\; n\ge0,\\
&\eqref{eq:first_part_of_relations}:\;
[t,s]=s,\; [t,e]=2e,\; [t,X^0]=-X^0,\\
&\eqref{eq:B-relations}:\;
[s,e]=0,\; [e,X^0]=-\tfrac32X^2,\; [X^0,X^1]=0
\end{aligned}
\right.
\presentationrangle_{\mathsf{Lie},\Zf}.
\]

All the relations from~\eqref{eq:lie-auxiliary-brackets} follow. Since
$[s,e]=0$, the operators $\ad_s$ and $\ad_e$ commute, and hence
\[
[e,X^n]=(\ad_s)^n[e,X^0]= -\tfrac32 X^{n+2}.
\]
The Jacobi identity and $[t,s]=s$ give, by induction on $n$,
\[
[t,X^n]=(n-1)X^n.
\]
For all $i,j\ge0$, the derivation rule yields
\[
\bigl(\tfrac12(\ad_s)^2+\tfrac13\ad_e\bigr)[X^i,X^j]
=[X^{i+1},X^{j+1}].
\]
Starting from $[X^0,X^1]=0$, this gives $[X^i,X^{i+1}]=0$ for all $i$.
Moreover, Jacobi gives
\[
0=[s,[X^i,X^{j-1}]]=[X^{i+1},X^{j-1}]+[X^i,X^j],
\]
so induction on $j-i$ shows that $[X^i,X^j]=0$ for every $i,j\ge0$.
\end{proof}

Let $P$ be the Lie subring of $\mathrm{FL}_\mathbb{Q}(t, a)$ generated by the elements $t, a, f_1/5, f_2/5$.

\begin{lemma}\label{lem:B-identity}
In $\FL_\Z(t,a)$,
\[
\begin{aligned}
3f_2 + (\delta^3 - 3 \delta^2 - \delta) f_2 &= (\delta+1)(\delta-1)(\delta-3)f_2 \\
&=(\delta+1)\bigl(\delta[a,f_1]+[w,f_1]\bigr) - (\delta-3)\bigl(3[\delta w,w,f_1]+[w,w,\delta f_1]\bigr).
\end{aligned}
\]
\end{lemma}
\begin{proof}
Direct computation.
\end{proof}

\begin{proposition}\label{prop:B-parafree}
The homomorphism $F=\FL_\Z(\xi_1,\xi_2)\to P$ given by
$\xi_1\mapsto t$, $\xi_2\mapsto a-2f_1/5$ induces isomorphisms
$F/\gamma_n(F)\cong P/\gamma_n(P)$ for all $n\geq1$.
In particular, $P$ is parafree of rank two.
\end{proposition}

\begin{proof}
Let $x=a-\frac{2f_1}{5} \in P$. We claim that
\begin{equation}
\label{eq:P_generation}
P=\Z t+\Z x+[P,P].
\end{equation}
Note that
\[
f_1-2a=[t,t,t,a]-2[t,t,a]-[t,a]\in[P,P].
\]
Hence, $a = 5x+2(f_1 - 2a)$ and $f_1/5 = 2x + (f_1 - 2a)$ belong to the right hand side of~\eqref{eq:P_generation}.
Dividing the identity in Lemma~\ref{lem:B-identity} by $5$ shows that
\[\frac{3f_2}{5}+(\delta^3-3\delta^2-\delta) \frac{f_2}{5}\in[P,P].\]
The nonconstant terms in
this expression are brackets with $t$, so $3f_2/5\in[P,P]$. Also
$f_2\in[P,P]$, whence
\[f_2/5=2 \cdot \frac{3f_2}{5} - f_2\in[P,P].\]
It follows that
\[
P=\Z t+\Z x+[P,P].
\]
The linear parts of $t,x$ in $\FL_\Q(t,a)$ are $t,a/5$. As in the proof
of Proposition~\ref{prop:thm-A-identification}, the generation equality
and these independent linear parts show that
$\FL_\Z(\xi_1,\xi_2)\to P$, $\xi_1\mapsto t$, $\xi_2\mapsto x$,
induces isomorphisms on all lower central quotients. Since $P$ is a Lie
subring of a free Lie algebra,
it is residually nilpotent.
\end{proof}

It remains to show that $H_2(P)$ is infinite. We do this by proving that
$H_2(G\otimes\F_5)=0$ and $H_3(G\otimes\F_5)$ is infinite-dimensional, and
then applying Proposition~\ref{prop:lie-divided-relators}.

\begin{lemma}\label{lem:weights-hs}
Let $N$ be a Lie algebra over a field $\kk$ and let $\delta$ be a derivation
of $N$. The induced chain map on the Chevalley--Eilenberg complex induces
an operator $\delta_*$ on each $H_q(N)$. If $\delta_*$ is diagonalizable on every vector space $H_q(N)$,
then
\[
H_p(N\rtimes_\delta\kk t)
\cong H_p(N)_0\oplus H_{p-1}(N)_0,
\]
where the subscript $0$ denotes the zero eigenspace.
\end{lemma}

\begin{proof}
The Hochschild--Serre homology spectral sequence for
$0\to N \to N \rtimes_\delta \kk t \to \kk t\to0$ has
\[
E^2_{p,q}=H_p(\kk t,H_q(N))\Longrightarrow H_{p+q}(N\rtimes_\delta\kk t).
\]
Since $\kk t$ is one-dimensional, only the columns $p=0,1$ are nonzero.
For $p\ge0$, there is a short
exact sequence
\[
0\to\operatorname{coker}\bigl(\delta_*\colon H_p(N)\to H_p(N)\bigr)
\to H_p(N\rtimes_\delta\kk t)
\to\ker\bigl(\delta_*\colon H_{p-1}(N)\to H_{p-1}(N)\bigr)
\to0,
\]
where $H_{-1}(N)=0$.

When $\delta_*$ is diagonalizable, there
is the eigenspace decomposition
\[H_p(N) = \bigoplus_{r \in \kk} H_p(N)_r, \quad H_p(N)_r = \{v \in H_p(N) \mid \delta_*(v) = rv\}.\]
So, the cokernel and kernel of $\delta_*$ are both identified with the zero eigenspace.
\end{proof}

\begin{proposition}\label{prop:B-homology}
The Lie algebra $G$ satisfies $H_2(G\otimes\F_5)=0$, and
$H_3(G\otimes\F_5)$ is infinite-dimensional.
\end{proposition}

\begin{proof}
Put $M=\F_5[X]$ and $H=\F_5s\oplus\F_5e$. Then
$G\otimes\F_5 = (M \rtimes H) \rtimes \F_5 t$.

The basis elements $s,e,X^i$ of $M\rtimes H$ are eigenvectors for $\ad t$
with weights $1,2,i-1$, respectively. Hence exterior products of basis
elements are eigenvectors in the Chevalley--Eilenberg complex. Its
differential preserves weights because $\ad t$ is a derivation, so the
complex and its homology decompose into weight spaces. In particular, the
induced action on each $H_q(M\rtimes H)$ is diagonalizable.
Lemma~\ref{lem:weights-hs} gives the following isomorphisms of
$\F_5$-vector spaces:
\begin{align*}
H_2(G \otimes \F_5) &\cong H_2(M \rtimes H)_0 \oplus H_1(M \rtimes H)_0, \\
H_3(G \otimes \F_5) &\cong H_3(M \rtimes H)_0 \oplus H_2(M \rtimes H)_0.
\end{align*}

Consider the Chevalley--Eilenberg complex with trivial coefficients,
$\mathsf{CE}_\bullet(M\rtimes H;\F_5)=\Lambda^\bullet(M\rtimes H)$. Its degree-$n$
chain group decomposes as
\[
\mathsf{CE}_n(M\rtimes H;\F_5)=\bigoplus_{p+q=n}\Lambda^pH\otimes\Lambda^qM.
\]
Since $M$ and $H$ are abelian, the only
possibly nonzero brackets in the formula for the differential $d^{\mathsf{CE}}$ are $[H,M]\subseteq M$. Such a bracket replaces one $H$-factor
and one $M$-factor by one $M$-factor, so it preserves $q$. Hence the complex
splits by $q$.
Since $H$ is two-dimensional, the fixed-$q$ complex is
the Koszul complex for two commuting operators $s$ and $e$:
\begin{equation}\label{eq:koszul-fixed-q}
\begin{tikzcd}[cramped,row sep=2.25em]
	0 & {\Lambda^2H \otimes \Lambda^q M} & {\Lambda^1 H \otimes \Lambda^q M} & {\Lambda^0 H \otimes \Lambda^q M} & 0 \\
	& \Lambda^q M & {\Lambda^q M \oplus \Lambda^q M} & \Lambda^q M.
	\arrow[from=1-1, to=1-2]
	\arrow[from=1-2, to=1-3]
	\arrow["{=}"{marking, allow upside down}, draw=none, from=1-2, to=2-2]
	\arrow[from=1-3, to=1-4]
	\arrow["{=}"{marking, allow upside down}, draw=none, from=1-3, to=2-3]
	\arrow[from=1-4, to=1-5]
	\arrow["{=}"{marking, allow upside down}, draw=none, from=1-4, to=2-4]
	\arrow["{w \mapsto (-ew, sw)}"', from=2-2, to=2-3]
	\arrow["{(u, v) \mapsto su + ev}"', from=2-3, to=2-4]
\end{tikzcd}
\end{equation}
Here, for the derivation $h$ on $M$ the action of $h$ on $\Lambda^q M$ is given by
\[h \cdot (m_1 \wedge \ldots \wedge m_q) = \sum_{i=1}^q m_1 \wedge \ldots \wedge [h, m_i] \wedge \ldots \wedge m_q.\]
\begin{itemize}
\item For $q=0$, $\Lambda^0M=\F_5$ and both operators act trivially. 
  The
  homology is $\Lambda^0H \cong \F_5$ in degree $0$ (generated by $1$), $\Lambda^1 H \cong \F_5^{\oplus 2}$ in degree $1$ (generated by $s$ and $e$), and $\Lambda^2H \cong \F_5$
  in degree $2$ (generated by $s \wedge e$).
  \item For $q=1$, $\Lambda^1M=\F_5[X]$, and $s,e$ act by multiplication by
  $X$ and $-\frac32X^2$, respectively. The homology is $\F_5$ in
  Koszul degree $0$ (generated by $X^0$), $\F_5$ in degree $1$ (generated by
  $2e\wedge X^0+3s\wedge X^1$), and zero in degree $2$.
  \item For $q=2,3$, the determinant map $f_1 \wedge \ldots \wedge f_q \mapsto \det(f_i(X_j))$ identifies
  $\Lambda^q M$ with the subspace of alternating polynomials in $\F_5[X]^{\otimes q} = \F_5[X_1, \ldots, X_q]$. That is,
  \[\Lambda^q M \cong \Delta_q \F_5[\sigma_1, \ldots, \sigma_q],\]
  where $\sigma_i$ is the $i$th elementary symmetric polynomial, and $\Delta_q = \prod_{i < l}(X_l - X_i)$.
  Under this identification, $s,e$ act by
  multiplication by $\sum_i X_i = \sigma_1$ and
  $-\tfrac32 \sum_i X_i^2 = -\tfrac32(\sigma_1^2-2\sigma_2)$, respectively. These polynomials form a
  regular sequence, so the homology is concentrated in Koszul degree $0$ and is equal to the quotient
  \[\Delta_q \F_5[\sigma_1, \ldots, \sigma_q] / (\sigma_1, -\tfrac32(\sigma_1^2 - 2\sigma_2)) \cong \Delta_q\F_5[\sigma_1, \ldots, \sigma_q] / (\sigma_1, \sigma_2).\]
  For $q=2$ the homology is $\Delta_2 \F_5 \cong \F_5$:
  since the determinant map sends
  \[X^0 \wedge X^1 \mapsto \det\begin{pmatrix}
  1 & 1 \\
  X_1 & X_2
  \end{pmatrix} = X_2 - X_1 = \Delta_2,\]
  the homology group is generated by $X^0\wedge X^1$. For $q=3$ the
  homology is the infinite-dimensional space $\Delta_3\F_5[\sigma_3]$:
  the determinant map sends
  \[X^n \wedge X^{n+1} \wedge X^{n+2} \mapsto \det\begin{pmatrix}
  X_1^n & X_2^n & X_3^n \\
  X_1^{n+1} & X_2^{n+1} & X_3^{n+1} \\
  X_1^{n+2} & X_2^{n+2} & X_3^{n+2}
  \end{pmatrix} = (X_1X_2X_3)^n \Delta_3 = \sigma_3^n \Delta_3,\]
  a basis of the homology group is represented by
  $X^n\wedge X^{n+1}\wedge X^{n+2}$, $n\geq0$.
\end{itemize}

Adding the Koszul degree to $q$ gives the following basis of
$H_i(M\rtimes H)$ through degree $3$. The weights are those induced by
$\ad t$, taken modulo $5$.
\[
\begin{array}{c|l|l}
i & \text{basis of }H_i(M\rtimes H) & \text{weights} \\
\hline
0 & 1 & 0\\
1 & s,\ e,\ X^0 & 1,\ 2,\ -1\\
2 & s\wedge e,\; 2e\wedge X^0+3s\wedge X^1,\;
    X^0\wedge X^1 & 3,\ 1,\ -1\\
3 & X^n\wedge X^{n+1}\wedge X^{n+2}, \; n\geq0 & 3n, \; n\geq0
\end{array}
\]
Thus $H_1(M\rtimes H)_0=H_2(M\rtimes H)_0=0$. In degree $3$, the
classes with $5\mid n$ have weight zero, so $H_3(M\rtimes H)_0$ is
infinite-dimensional.
\end{proof}

\begin{proposition}\label{prop:B-P-homology}
The group $H_2(P)$ is a $5$-primary torsion group, and
\[
H_2(P)\otimes\F_5\cong\bigoplus_{j\ge0}\F_5.
\]
\end{proposition}

\begin{proof}
Apply Proposition~\ref{prop:lie-divided-relators} with $p=5$ and relators
$f_1,f_2$. Proposition~\ref{prop:B-presentation} identifies the quotient
$G$ required there with the free $\Zf$-module described above, and the
resulting Lie subring $D$ is $P$. By Proposition~\ref{prop:B-parafree},
$P_{\mathrm{ab}}\cong\Z^2$, so
Proposition~\ref{prop:lie-divided-relators}, together with
Proposition~\ref{prop:B-homology}, gives
\[
H_2(P)\otimes\F_5\cong H_3(G\otimes\F_5).
\]
The latter is at most countably-dimensional, since $G\otimes\F_5$ is
generated by finitely many elements and hence its Chevalley--Eilenberg
chain groups are at most countable-dimensional. Proposition~\ref{prop:B-homology}
shows it is infinite-dimensional, so it is isomorphic to
$\bigoplus_{j\ge0}\F_5$.
After inverting $5$, $P\otimes\Z[1/5]=\FL_{\Z[1/5]}(t,a)$, so
$H_2(P)\otimes\Z[1/5]=0$ by~\eqref{eq:lie-free-resolution}. Thus
$H_2(P)$ is a $5$-primary torsion group.
\end{proof}

\begin{proof}[Proof of Theorem~\ref{thm:B}]
Proposition~\ref{prop:B-parafree} proves the parafreeness assertion, and
Proposition~\ref{prop:B-P-homology} gives the stated computation of
$H_2(P)\otimes\F_5$.

By Lemma~\ref{lem:lie-fp-homology}, finite presentability of $P$ would imply
that $H_2(P)$ is finitely generated. This contradicts the preceding
calculation, so $P$ is not finitely presented.
\end{proof}

\section{Preliminaries on augmented algebras}\label{sec:preliminaries}

\subsection{Completion}

Every augmented algebra $B$ is filtered by the powers of its augmentation
ideal:
\[
 B=I(B)^0\supseteq I(B)\supseteq I(B)^2\supseteq\cdots.
\]
Its completion is
\[
 \widehat B=\varprojlim_n B/I(B)^n.
\]
Every morphism $f\colon B\to C$ of augmented algebras satisfies
\[
 f\bigl(I(B)^n\bigr)\subseteq I(C)^n,
 \qquad n\geq0,
\]
and therefore induces a homomorphism
\(\widehat f\colon\widehat B\to\widehat C\).

\subsection{Monoid algebras}

Let $M$ be a monoid. Its \emph{monoid algebra} is the free $\Bbbk$-module
\[
 \Bbbk[M]=\bigoplus_{w\in M}\Bbbk w
\]
with multiplication
\[
 \left(\sum_{w\in M}\alpha_w w\right)
 \left(\sum_{w\in M}\beta_w w\right)
 =\sum_{w\in M}\left(\sum_{uv=w}\alpha_u\beta_v\right)w.
\]
It has the canonical augmentation
$\sum_{w\in M}\alpha_w w\mapsto\sum_{w\in M}\alpha_w$.

\begin{remark}\label{rem:free-monoid-algebra}
If $M=X^*$ is the free monoid on a set $X$, then
\[
 \Bbbk\langle X\rangle\xrightarrow{\sim}\Bbbk[M],
 \qquad x\mapsto x-1,
\]
is an isomorphism of augmented algebras, where
$\Bbbk\langle X\rangle$ is augmented by $x\mapsto0$.
\end{remark}

\begin{lemma}\label{lem:free-monoid-subalgebra}
Let $X$ be a set, and augment $\Bbbk\langle X\rangle$ by $x\mapsto0$ for
$x\in X$. Every augmented subalgebra of $\Bbbk\langle X\rangle$ is
residually nilpotent. In particular, if $M\subseteq X^*$ is a submonoid,
then $\Bbbk[M]$, augmented by $w\mapsto1$, is residually nilpotent.
\end{lemma}

\begin{proof}
If $B\subseteq\Bbbk\langle X\rangle$ is an augmented subalgebra, then
\(I(B)^n\subseteq(X)^n\).
The powers of $(X)$ have zero intersection by word length, so
$I(B)^\omega=0$. For the final assertion, the inverse of the isomorphism in
Remark~\ref{rem:free-monoid-algebra} embeds $\Bbbk[M]$ as an augmented
subalgebra of $\Bbbk\langle X\rangle$.
\end{proof}

For a monoid homomorphism $\varphi\colon M\to N$, its kernel congruence is
\[
 M\times_NM
 =\{(u_1,u_2)\in M\times M\mid \varphi(u_1)=\varphi(u_2)\}.
\]
Since the category of monoids is monadic over $\mathbf{Set}$, the canonical map onto the image is the coequalizer
\[
 M\times_NM
 \overset{\pi_1}{\underset{\pi_2}{\rightrightarrows}}
 M\xrightarrow{\varphi}\varphi(M)\subseteq N.
\]

\begin{lemma}\label{lem:algebra-kernel}
Let $\varphi\colon M\to N$ be a monoid homomorphism whose kernel congruence is generated by a set $\{(u_\lambda,v_\lambda)\}_{\lambda\in\Lambda}$.
Then the kernel of the induced homomorphism of augmented algebras is the two-sided ideal
\[
 \ker\bigl(\varphi_*\colon\Bbbk[M]\to\Bbbk[N]\bigr)
 =(u_\lambda-v_\lambda\mid\lambda\in\Lambda).
\]
\end{lemma}

\begin{proof}
Let $J=(u_\lambda-v_\lambda\mid\lambda\in\Lambda)$. The relation
\[
 u\sim_Jv\quad\Longleftrightarrow\quad u-v\in J
\]
is a monoid congruence containing every $(u_\lambda,v_\lambda)$. Hence
$\varphi(u)=\varphi(v)$ implies $u-v\in J$. The decomposition
\[
 \Bbbk[M]=\bigoplus_{w\in N}\Bbbk[\varphi^{-1}(w)]
\]
makes $\Bbbk[M]$ an $N$-graded $\Bbbk$-algebra. On the component of degree
$w$, the map $\varphi_*$ is the coefficient-sum map to $\Bbbk w$. Hence
$\ker\varphi_*$ is spanned over $\Bbbk$ by such differences. Thus
$\ker\varphi_*\subseteq J$; the reverse inclusion is immediate.
\end{proof}

\subsection{2-acyclic morphisms}

A morphism $B\to C$ of augmented algebras is called \emph{$2$-acyclic} if
\mbox{$H_1(B)\to H_1(C)$} is an isomorphism and
\mbox{$H_2(B)\to H_2(C)$} is surjective.

By \cite[Proposition~4.4]{IL21}, every $2$-acyclic morphism $B\to C$
induces isomorphisms $B/I(B)^n\cong C/I(C)^n$ for all $n\geq1$.

Let $B$ be an augmented algebra and let $X$ be a set. Augment
$\Bbbk\langle X\rangle$ by $x\mapsto0$. The free product
$B*\Bbbk\langle X\rangle$ has the augmentation induced by those of its
factors. Its elements are called \emph{$B$-polynomials}. A $B$-polynomial is
\emph{acyclic} if it lies in the kernel of the morphism
\[
 B*\Bbbk\langle X\rangle\to
 \Bbbk\langle X\rangle/(X)^2
\]
that annihilates $I(B)$ and sends each $x\in X$ to its residue class.

For a family $P=(p_x)_{x\in X}$ of acyclic $B$-polynomials, set
\[
 B_P=(B*\Bbbk\langle X\rangle)/(x-p_x\mid x\in X).
\]

\begin{lemma}\label{lem:standard-acyclic}
The canonical morphism $B\to B_P$ is $2$-acyclic. In particular, it induces
isomorphisms $B/I(B)^n\cong B_P/I(B_P)^n$ for all $n\geq1$.
\end{lemma}

\begin{proof}
See \cite[Proposition~4.6]{IL21}.
\end{proof}

\subsection{Finitely presented algebras}

\begin{lemma}\label{lem:fp-homology}
Let $R$ be a commutative ring, and let $B$ be a finitely presented
augmented $R$-algebra. Then $H_2(B)=\operatorname{Tor}^B_2(R,R)$ embeds
into a finitely generated $R$-module. If $R$ is Noetherian, then $H_2(B)$
is finitely generated as an $R$-module; in particular, it is
finite-dimensional when $R$ is a field.
\end{lemma}

\begin{proof}
After shifting a finite generating set into $I(B)$, choose a finite
augmented presentation $B=F/J$, where $F=R\langle x_1,\ldots,x_n\rangle$
and $J$ is generated as an ideal by $r_1,\ldots,r_s$. Since
$J\subseteq I(F)$, the Hopf formula~\cite[Corollary~4.2]{IL21} gives
\[
 H_2(B)\cong
 \frac{J\cap I(F)^2}{I(F)J+JI(F)}.
\]
The $R$-module $J/(I(F)J+JI(F))$ is generated by the classes of
$r_1,\ldots,r_s$, because
\[
 [ur_iv]=\varepsilon_F(u)\varepsilon_F(v)[r_i],
 \qquad u,v\in F.
\]
Thus $H_2(B)$ embeds into this finitely generated module. If $R$ is
Noetherian, its submodules are finitely generated, proving the second
assertion.
\end{proof}

\section{Identifying \texorpdfstring{$A$}{A} with a subalgebra of a free algebra}\label{sec:monoid}

Let $L$ be the submonoid of the free monoid $\{X_1,X_2,X_3,X_4\}^*$ generated by
the six-element set
\[
 \{X_1,X_2X_3,X_2X_4,X_1X_2,X_3X_2,X_4\}.
\]
It is a well-known example of a finitely generated submonoid of a free monoid
that is not finitely presented; see \cite[Example~3.7]{SZ18}. In this section
we show that the algebra $A$ given by \eqref{eq:def-A} is isomorphic to the
monoid algebra $\Bbbk[L]$. In particular, Lemma~\ref{lem:free-monoid-subalgebra}
will imply that $A$ is residually nilpotent.

Define a homomorphism of free monoids
\[
 \theta\colon\{Y_1,\ldots,Y_6\}^*\to
 \{X_1,X_2,X_3,X_4\}^*
\]
by
\[
 \begin{aligned}
 \theta(Y_1)&=X_1, & \theta(Y_2)&=X_2X_3,& \theta(Y_3)&=X_2X_4,\\
 \theta(Y_4)&=X_1X_2,& \theta(Y_5)&=X_3X_2,& \theta(Y_6)&=X_4.
 \end{aligned}
\]

\begin{lemma}\label{lem:monoid-kernel}
The kernel congruence of $\theta$ is generated by
\[
 Y_1Y_2^\ell Y_3=Y_4Y_5^\ell Y_6,
 \qquad \ell\geq0.
\]
\end{lemma}

\begin{proof}
The words on the left-hand sides have the common image
\[
 X_1(X_2X_3)^\ell X_2X_4
 =X_1X_2(X_3X_2)^\ell X_4.
\]
Let $\equiv$ be the congruence generated by the relations in the statement. We prove
$\theta(u)=\theta(v)\Rightarrow u\equiv v$ by induction on the length of the
common image. The empty image is immediate. If $u$ and $v$ have the same
first factor, cancel it and apply induction. Otherwise, let their first
factors be $Y_i$ and $Y_j$, where $i\neq j$. If
\[
 u=Y_i u',
 \qquad
 v=Y_j v',
\]
then
\[
 \theta(Y_i)\theta(u')=\theta(Y_j)\theta(v').
\]
Thus one of $\theta(Y_i)$ and $\theta(Y_j)$ is a prefix of the other: any two
prefixes of the same word are comparable under inclusion, and the shorter is
a prefix of the longer. In the set
$\{X_1,X_2X_3,X_2X_4,X_1X_2,X_3X_2,X_4\}$, the only distinct words with this
property are $X_1$ and $X_1X_2$. Hence
$\{Y_i,Y_j\}=\{Y_1,Y_4\}$.

After interchanging $u$ and $v$ if necessary, one has
\[
 u=Y_1Y_2^\ell Y_3u'',
 \qquad
 v=Y_4Y_5^\ell Y_6v'',
 \qquad
 \theta(u'')=\theta(v'')
\]
for some $\ell\geq0$. Indeed, cancelling the initial $X_1$ gives
$\theta(u')=X_2\theta(v')$, so the next factor of $u'$ is $Y_2$ or $Y_3$. In
the first case the next factor of $v'$ is necessarily $Y_5$ and the comparison
repeats; in the second it is necessarily $Y_6$ and the remaining images are
equal. Since the words are finite, the second case must occur.
Therefore
\[
 u\equiv Y_4Y_5^\ell Y_6u''
 \equiv Y_4Y_5^\ell Y_6v''=v,
\]
where the second equivalence follows by induction.
\end{proof}

Let
\[
 \begin{aligned}
 \varphi&\colon\Bbbk\langle a,b,c,d,x,y\rangle
 \to\Bbbk[\{Y_1,Y_2,Y_3,Y_4,Y_5,Y_6\}^*],\\
 \psi&\colon\Bbbk\langle x_1,x_2,x_3,x_4\rangle
 \to\Bbbk[\{X_1,X_2,X_3,X_4\}^*]
 \end{aligned}
\]
be the isomorphisms of augmented algebras that are given, respectively, by
\begin{equation}\label{eq:variables}
 \begin{aligned}
 (a,b,c,d,x,y)&\mapsto
 (Y_1-1,Y_2-1,Y_3-1,Y_4-1,Y_5-1,Y_6-1),\\
 (x_1,x_2,x_3,x_4)&\mapsto
 (X_1-1,X_2-1,X_3-1,X_4-1).
 \end{aligned}
\end{equation}
Note that
\[
 \begin{aligned}
 \varphi(e_\ell)
 &\overset{\mathclap{\eqref{eq:def-e}}}{=}
 \varphi\bigl((1+a)(1+b)^\ell(1+c)-(1+d)(1+x)^\ell(1+y)\bigr)\\
 &=Y_1Y_2^\ell Y_3-Y_4Y_5^\ell Y_6,
 \qquad \ell\geq0.
 \end{aligned}
\]
\begin{corollary}\label{cor:monoid}
The algebra $A$ given by \eqref{eq:def-A} is isomorphic to $\Bbbk[L]$ as an
augmented algebra.
\end{corollary}

\begin{proof}
Lemmas~\ref{lem:monoid-kernel} and~\ref{lem:algebra-kernel}, together with the
preceding computation, give
\[
 \ker(\theta_*\circ\varphi)=(e_\ell\mid\ell\geq0).
\]
Thus $A\cong\operatorname{im}(\theta_*\circ\varphi)=\Bbbk[L]$.
\end{proof}

\begin{samepage}
\begin{remark}\label{rem:free-embedding}
The composite
\[
 \Bbbk\langle a,b,c,d,x,y\rangle
 \xrightarrow{\varphi}\Bbbk[\{Y_1,\ldots,Y_6\}^*]
 \xrightarrow{\theta_*}\Bbbk[\{X_1,\ldots,X_4\}^*]
 \xrightarrow{\psi^{-1}}\Bbbk\langle x_1,\ldots,x_4\rangle
\]
acts by
\begin{equation}\label{eq:free-embedding-map}
 \begin{aligned}
  a&\mapsto x_1,&
  b&\mapsto x_2+x_3+x_2x_3,&
  c&\mapsto x_2+x_4+x_2x_4,\\
  d&\mapsto x_1+x_2+x_1x_2,&
  x&\mapsto x_2+x_3+x_3x_2,&
  y&\mapsto x_4.
 \end{aligned}
\end{equation}
The correspondence \eqref{eq:free-embedding-map} identifies $A$ with a
subalgebra of the free algebra
$\Bbbk\langle x_1,x_2,x_3,x_4\rangle$.
\end{remark}
\end{samepage}

\section{Parafreeness of \texorpdfstring{$A$}{A}}\label{sec:completion}

In this section we prove that the algebra $A$ given by \eqref{eq:def-A} is
parafree of rank four via the canonical map
$\Bbbk\langle a,b,c,d\rangle\to A$.

For $\ell\geq0$, set
\begin{equation}\label{eq:def-f}
 f_\ell=(1+a)b^\ell(1+c)-(1+d)x^\ell(1+y).
\end{equation}

\begin{lemma}\label{lem:change-relations}
For every $\ell\geq0$,
\[
 f_\ell=\sum_{j=0}^\ell(-1)^{\ell-j}\binom{\ell}{j}e_j,
 \qquad
 e_\ell=\sum_{j=0}^\ell\binom{\ell}{j}f_j.
\]
Consequently,
\[
 (e_\ell\mid\ell\geq0)=(f_\ell\mid\ell\geq0),
 \qquad
 (e_0,e_1)=(f_0,f_1).
\]
\end{lemma}

\begin{proof}
The first formula follows from \eqref{eq:def-e} and the binomial identity
\[
 \sum_{j=0}^\ell(-1)^{\ell-j}\binom{\ell}{j}(1+u)^j=u^\ell
\]
by substituting $u=b$ and $u=x$. The second formula follows by binomial
inversion.
\end{proof}

Thus, for the algebras $A$ and $A_0$ given by \eqref{eq:def-A} and
\eqref{eq:def-A0}, respectively,
\[
 A=\Bbbk\langle a,b,c,d,x,y\rangle/(f_\ell\mid\ell\geq0),
 \qquad
 A_0=\Bbbk\langle a,b,c,d,x,y\rangle/(f_0,f_1).
\]

The first two relations are
\begin{equation}\label{eq:first-relations}
 \begin{aligned}
 f_0&=a+c-d-y+ac-dy,\\
 f_1&=b-x+ab+bc-dx-xy+abc-dxy.
 \end{aligned}
\end{equation}

\begin{lemma}\label{lem:elimination}
Let $a,b,c,d,x,y$ be elements of a unital algebra such that $1+a$, $1+c$,
and $1+d$ are invertible. Set $h=(1+a)^{-1}(1+d)$. Then $f_0=f_1=0$ if
and only if
\begin{equation}\label{eq:elimination}
 x=h^{-1}bh,
 \qquad
 1+y=h^{-1}(1+c).
\end{equation}
The identities \eqref{eq:elimination} imply $f_\ell=0$ for all $\ell\geq0$.
\end{lemma}

\begin{proof}
The identities
\[
 f_0=(1+a)\bigl((1+c)-h(1+y)\bigr)
\]
and, when $f_0=0$,
\[
 f_1=(1+a)(b-hxh^{-1})(1+c)
\]
give the equivalence. Under \eqref{eq:elimination}, one has
$x^\ell=h^{-1}b^\ell h$, and substitution in \eqref{eq:def-f} gives
$f_\ell=0$.
\end{proof}

\begin{proposition}\label{prop:A_is_parafree}
Let $F=\Bbbk\langle a,b,c,d\rangle$. Then:
\begin{enumerate}
\item $I(A)^\omega = 0$;
\item $\ker(A_0 \to A) = I(A_0)^\omega$;
\item the map $F \to A_0$ induces isomorphisms $F/I(F)^n\cong A_0/I(A_0)^n$
for all $n\geq1$;
\item the projection $A_0 \twoheadrightarrow A$ induces isomorphisms
$A_0/I(A_0)^n\cong A/I(A)^n$ for all $n\geq1$.
\end{enumerate}
In particular, $A$ is parafree of rank $4$.
\end{proposition}
\begin{proof}
(1) By Corollary~\ref{cor:monoid}, $A \cong \Bbbk[L]$, where $L$ is a submonoid of a free monoid. Hence, $A$ is residually nilpotent by Lemma~\ref{lem:free-monoid-subalgebra}.

(2) Set $K=\ker(A_0\twoheadrightarrow A)$. By Lemma~\ref{lem:change-relations},
$K=(f_\ell\mid\ell\geq2)$. In $\widehat{A_0}$, the elements $1+a$, $1+c$,
and $1+d$ are invertible. Lemma~\ref{lem:elimination} therefore shows that
every $f_\ell$ vanishes in $\widehat{A_0}$. Hence
\[
 K\subseteq\ker(A_0\to\widehat{A_0})=I(A_0)^\omega.
\]
Conversely, the image of $I(A_0)^\omega$ in $A$ lies in $I(A)^\omega=0$. Thus $K=I(A_0)^\omega$.

(3) In the presentation \eqref{eq:def-A0}, the right-hand sides defining
$x$ and $y$ are acyclic $F$-polynomials. Hence the assertion follows from
Lemma~\ref{lem:standard-acyclic}.

(4) For every $n\geq1$, the projection $A_0\twoheadrightarrow A$ maps
$I(A_0)^n$ onto $I(A)^n$ and has kernel $K\subseteq I(A_0)^n$. It therefore
induces an isomorphism $A_0/I(A_0)^n\xrightarrow{\sim}A/I(A)^n$.
\end{proof}

\begin{remark}
By Remark~\ref{rem:free-embedding}, formula~\eqref{eq:free-embedding-map}
identifies $A$ with a subalgebra of
$F=\Bbbk\langle x_1,x_2,x_3,x_4\rangle$. The natural composite
\[
 \Bbbk\langle a,b,c,d\rangle\hookrightarrow A\hookrightarrow F
\]
is the restriction of \eqref{eq:free-embedding-map} to $a,b,c,d$. Its images
have linear parts $x_1,x_2+x_3,x_2+x_4,x_1+x_2$, which form a basis of
$H_1(F)$.
Comparing lowest nonzero homogeneous components shows that the composite is
injective, so the first arrow is also an embedding. Since the composite
induces an isomorphism on $H_1$ and both algebras have zero $H_2$, it is
$2$-acyclic. Hence, by \cite[Proposition~4.4]{IL21}, it induces
isomorphisms on all quotients by powers of the augmentation ideals.
\end{remark}

\section{Proofs of Theorems~\ref{thm:main} and~\ref{thm:transfinite-length}}\label{sec:homology}

Let $A$ and $A_0$ be the algebras given by \eqref{eq:def-A} and
\eqref{eq:def-A0}.
In this section we prove that $H_2(A)$ is a free $\Bbbk$-module of
countably infinite rank.

Set $K = \ker(A_0\twoheadrightarrow A)$ and $I = I(A_0)$.

\begin{proposition}\label{prop:identification_of_H_2_A}
There is a natural isomorphism
\[
 H_2(A)\xrightarrow{\sim}
 \frac{K}
 {I \cdot K + K \cdot I}
\]
sending the class of $f_\ell$ to its residue class for every $\ell\geq2$.
\end{proposition}

\begin{proof}
Since \(F = \Bbbk \langle a, b, c, d \rangle \to A_0\) is $2$-acyclic, the surjection
$0 = H_2(F)\twoheadrightarrow H_2(A_0)$ gives $H_2(A_0)=0$.

By
Proposition~\ref{prop:A_is_parafree}~(2), $K = I(A_0)^\omega \subseteq I(A_0)^2$. Hence \cite[Lemma~4.1]{IL21} gives an exact sequence
\[
0 = H_2(A_0)\to H_2(A)\to
 \frac{K}
 {I(A_0)\cdot K + K \cdot I(A_0)}\to0,
\]
hence the asserted isomorphism.
\end{proof}

Retain the elements $Y_1,\ldots,Y_6$ from
\eqref{eq:variables}, introduce variables
$z_\ell$, $\ell\geq2$, and send
\[
 z_\ell\mapsto e_\ell=Y_1Y_2^\ell Y_3-Y_4Y_5^\ell Y_6.
\]
This gives the presentation
\begin{equation}\label{eq:z-presentation}
 A_0\cong
 \frac{\Bbbk\langle Y_1,\ldots,Y_6,z_2,z_3,\ldots\rangle}
 {(e_0,e_1,e_\ell-z_\ell\mid\ell\geq2)}.
\end{equation}
Here the augmentation is given by $Y_i\mapsto1$ and $z_\ell\mapsto0$.
Thus $K=(z_\ell\mid\ell\geq2)$.

\begin{lemma}\label{lem:normal-bases}
\leavevmode
\begin{enumerate}
\item A $\Bbbk$-module basis of $A$ consists of the words
$w\in\{Y_1,\ldots,Y_6\}^*$ with no subword $Y_1Y_2^jY_3$, $j\geq0$.
\item A $\Bbbk$-module basis of $A_0$ consists of the words
\[
 w_0z_{\ell_1}w_1\cdots z_{\ell_t}w_t,
 \qquad t\geq0,
\]
where $\ell_i\geq2$, and each $w_i\in\{Y_1,\ldots,Y_6\}^*$ has no subword
$Y_1Y_2^jY_3$, $j\geq0$.
\end{enumerate}
\end{lemma}

\begin{proof}
For $A$, use the reductions
\[
 Y_1Y_2^\ell Y_3\rightsquigarrow Y_4Y_5^\ell Y_6,
 \qquad \ell\geq0.
\]
For $A_0$, use
\[
 \begin{aligned}
 Y_1Y_3&\rightsquigarrow Y_4Y_6,\\
 Y_1Y_2Y_3&\rightsquigarrow Y_4Y_5Y_6,\\
 Y_1Y_2^\ell Y_3&\rightsquigarrow Y_4Y_5^\ell Y_6+z_\ell,
 &&\ell\geq2.
 \end{aligned}
\]
Every reduction strictly decreases the number of occurrences of $Y_1$ in
each resulting monomial. The left-hand sides have neither proper inclusions
nor overlaps. The Diamond Lemma~\cite[Theorem~1.2]{Ber78} therefore gives the
stated bases.
\end{proof}

\begin{proposition}\label{prop:homology}
There is an isomorphism
\[
 K / (I \cdot K+K \cdot I)\cong \bigoplus_{\ell\geq2}\Bbbk z_\ell.
\]
\end{proposition}

\begin{proof}
By Lemma~\ref{lem:normal-bases}, the words \(w_0z_{\ell_1}w_1\cdots z_{\ell_t}w_t\) with $t \ge 1$, where $\ell_i \ge 2$, and each $w_i$ has no subwords $Y_1 Y_2^j Y_3$, form a $\Bbbk$-module basis of $K$.

Define a $\Bbbk$-linear map
$\pi\colon K\to \bigoplus_{\ell\geq2}\Bbbk z_\ell$ by
\[
 \pi(w_0z_{\ell_1}w_1\cdots z_{\ell_t}w_t)=
 \begin{cases}
  z_{\ell_1},&t=1,\\
  0,&t\geq2.
 \end{cases}
\]
First, we claim that $\pi$ vanishes on $I \cdot K + K \cdot I$. By
linearity, it is enough to fix a normal basis word $k$ of $K$, take an
arbitrary $a\in I$, and show that $\pi(ak)=\pi(ka)=0$. Write
\[
 k=w_0z_{\ell_1}w_1\cdots z_{\ell_t}w_t,
 \qquad t\geq1.
\]
Every reduction in Lemma~\ref{lem:normal-bases} preserves or increases the
number of $z$-letters in each term. Hence, if $t\geq2$, both $ak$ and $ka$
reduce to sums of normal words with at least two $z$-letters, and
$\pi(ak)=\pi(ka)=0$.

It remains to consider $t=1$. Write $a=x+y$, where $x$ is a linear
combination of normal basis words containing at least one $z$-letter and $y$
is a linear combination of normal basis words in the $Y_i$. Every term of
$xk$ and $kx$ has at least two $z$-letters, so
$\pi(xk)=\pi(kx)=0$. Since $a\in I$ and $\varepsilon_{A_0}(x)=0$, we also
have $\varepsilon_{A_0}(y)=0$. For any normal basis word $w$ in the $Y_i$,
\[
 \pi(wk)=\pi(k)=\pi(kw),
\]
because in reducing either product exactly one branch retains a single
$z$-letter, with the same index $\ell$; every other branch has at least two.
By linearity,
\[
 \pi(yk)= \pi(ky) = \varepsilon_{A_0}(y)\pi(k)=0.\]
Thus $\pi(ak)=\pi(ka)=0$, proving the claim.

Second, we claim that the classes $[z_\ell]$, $\ell\geq2$, span
$K/(I \cdot K + K \cdot I)$ as a $\Bbbk$-module. If $t \ge 2$, then
\([w_0z_{\ell_1}w_1\cdots z_{\ell_t}w_t] = 0\), since there is the decomposition \(w_0 z_{\ell_1} \cdot (w_1 z_{\ell_2} \ldots w_t) \in I \cdot K\).
If $t = 1$, then we use
\[
 w_0z_\ell w_1-z_\ell=(w_0-1)z_\ell w_1+z_\ell(w_1-1)
 \in I \cdot K+K \cdot I.
\]
Thus every basis word of $K$ is congruent modulo $I \cdot K+K \cdot I$ to
either zero or some $z_\ell$. The claim is proved.

By the first claim, $\pi$ induces a map on $K/(I \cdot K+K \cdot I)$ and
sends the class of each $z_\ell$ to the corresponding basis vector of
$\bigoplus_{\ell\geq2}\Bbbk z_\ell$. Hence these classes are
$\Bbbk$-linearly independent. Together with the second claim, this proves the asserted
isomorphism.
\end{proof}

\begin{proof}[Proof of Theorem~\ref{thm:main}]
Proposition~\ref{prop:A_is_parafree} proves that $A$ is parafree of rank four,
and Propositions~\ref{prop:identification_of_H_2_A} and~\ref{prop:homology}
give $H_2(A)\cong\Bbbk^{\oplus\aleph_0}$ as a $\Bbbk$-module.

Suppose that $A$ were finitely presented as a $\Bbbk$-algebra. Since
$\Bbbk\ne0$, choose a maximal ideal $\mathfrak m\subset\Bbbk$ and put
$K_0=\Bbbk/\mathfrak m$. Scalar extension preserves finite presentations,
so $A\otimes_\Bbbk K_0$ would be a finitely presented augmented $K_0$-algebra.
By Corollary~\ref{cor:monoid},
\[
 A\otimes_\Bbbk K_0\cong K_0[L].
\]
The homology computation above, applied over $K_0$, gives
$H_2(A\otimes_\Bbbk K_0)\cong K_0^{\oplus\aleph_0}$. This is
infinite-dimensional over $K_0$, contrary to Lemma~\ref{lem:fp-homology}.
Hence $A$ is not finitely presented.
\end{proof}

\subsection{Transfinite length of $\{I(A_0)^\alpha\}$}

\begin{lemma}\label{lem:kernel-powers}
For every $h\geq1$, one has $K^h\neq0$ and
\[
 \bigcap_{h\geq1}K^h=0.
\]
If $V=\bigoplus_{\ell\geq2}\Bbbk z_\ell$, then there is an isomorphism of
$A$-bimodules
\begin{equation}\label{eq:kernel-quotient}
 A\otimes_\Bbbk V\otimes_\Bbbk A\otimes_\Bbbk\cdots
 \otimes_\Bbbk V\otimes_\Bbbk A
 \cong
 \frac{K^h}{K^{h+1}},
\end{equation}
where the left-hand side has $h$ factors $V$. For $a_i\in A$ and lifts
$\widetilde a_i\in A_0$, the isomorphism is
\[
 a_0\otimes z_{\ell_1}\otimes a_1\otimes\cdots\otimes
 z_{\ell_h}\otimes a_h
 \mapsto
 [\widetilde a_0z_{\ell_1}\widetilde a_1\cdots
 z_{\ell_h}\widetilde a_h].
\]
\end{lemma}

\begin{proof}
The reductions in the proof of Lemma~\ref{lem:normal-bases} do not decrease
the number of variables $z_\ell$. Hence $K^h$ is contained in the span of the
normal words containing at least $h$ of them. Conversely, every such normal
word belongs to $K^h$. Thus these words form a basis of $K^h$, which gives
$z_2^h\neq0$ and $\bigcap_{h\geq1}K^h=0$.

Changing a lift $\widetilde a_i$ changes this product by an element
of $K^{h+1}$. With the bases from Lemma~\ref{lem:normal-bases}, the stated map
sends the tensor-product basis bijectively to the normal words containing
exactly $h$ variables $z_\ell$. Hence it is an $A$-bimodule isomorphism.
\end{proof}

\begin{lemma}\label{lem:bimodule-separation}
For every $h\geq1$,
\begin{equation}\label{eq:bimodule-separation}
 \bigcap_{k\geq0}\ \sum_{i+j=k}
 I(A)^i\frac{K^h}{K^{h+1}}I(A)^j=0.
\end{equation}
\end{lemma}

\begin{proof}
By Remark~\ref{rem:free-embedding}, regard $A$ as a subalgebra of
$F=\Bbbk\langle x_1,x_2,x_3,x_4\rangle$. Grouping the two outer factors in
\eqref{eq:kernel-quotient} gives
\[
 \begin{aligned}
 \frac{K^h}{K^{h+1}}
 &\overset{\eqref{eq:kernel-quotient}}{\cong}
 A\otimes_\Bbbk
 \bigl(V\otimes_\Bbbk A\otimes_\Bbbk\cdots\otimes_\Bbbk
 A\otimes_\Bbbk V\bigr)\otimes_\Bbbk A\\
 &\hookrightarrow
 F\otimes_\Bbbk
 \bigl(V\otimes_\Bbbk A\otimes_\Bbbk\cdots\otimes_\Bbbk
 A\otimes_\Bbbk V\bigr)\otimes_\Bbbk F,
 \end{aligned}
\]
where the last arrow is induced by $A\hookrightarrow F$ on the two outer
factors.
Since $I(F)=(x_1,x_2,x_3,x_4)$, under this embedding the
summand indexed by $i+j=k$ in \eqref{eq:bimodule-separation} has left outer
factor in $I(F)^i$ and right outer factor in $I(F)^j$. In the $x$-word
basis, the sum over $i+j=k$ is spanned by tensors for which the sum of the two
outer word lengths
is at least $k$. Since every element is a finite sum of basis tensors, the
intersection is zero. This proves \eqref{eq:bimodule-separation}.
\end{proof}

\begin{lemma}\label{lem:transfinite-terms}
For every $h\geq1$, one has $I^{\omega\cdot h}=K^h$.
\end{lemma}

\begin{proof}
The case $h=1$ is the definition of $K$. Suppose that
$I^{\omega\cdot h}=K^h$. Induction on $k$ gives
\[
 I^{\omega\cdot h+k}=\sum_{i+j=k}I^iK^hI^j,
 \qquad k\geq0.
\]
Since $K\subseteq I^k$, every sum in the formula above contains
$K^{h+1}\subseteq K^hI^k$. The action of $I$ on $K^h/K^{h+1}$ factors
through $I(A)$, and $I\to I(A)$ is surjective. Hence, for every $k\geq0$,
the quotient map $K^h\twoheadrightarrow K^h/K^{h+1}$ sends
$\sum_{i+j=k}I^iK^hI^j$ onto
\[
 \sum_{i+j=k}I(A)^i\frac{K^h}{K^{h+1}}I(A)^j.
\]
By Lemma~\ref{lem:bimodule-separation}, the intersection of these images is
zero. Since $K^{h+1}$ is contained in every such sum, an element
belonging to all these sums has, modulo $K^{h+1}$, an image in this zero
intersection. Hence it lies in $K^{h+1}$, and therefore
\[
 \bigcap_{k\geq0}\ \sum_{i+j=k}I^iK^hI^j=K^{h+1}.
\]
The ordinals $\omega\cdot h+k$, $k<\omega$, are cofinal in
$\omega\cdot(h+1)$, so $I^{\omega\cdot(h+1)}=K^{h+1}$. Induction on $h$
proves the assertion.
\end{proof}

\begin{proof}[Proof of Theorem~\ref{thm:transfinite-length}]
The ordinals $\omega\cdot h$, $h\geq1$, are cofinal in $\omega^2$.
Lemma~\ref{lem:transfinite-terms} and Lemma~\ref{lem:kernel-powers} give
\[
 I^{\omega^2}=\bigcap_{h\geq1}I^{\omega\cdot h}
 =\bigcap_{h\geq1}K^h=0.
\]

Every ordinal below $\omega^2$ is $\omega\cdot h+k$ for some $h,k<\omega$.
At a finite stage, $a^k\in I^k$ has image $x_1^k\neq0$ in $F$ by
\eqref{eq:free-embedding-map}. If $h\geq1$, then
\[
 I^{\omega\cdot h+k}\supseteq I^{\omega\cdot(h+1)}=K^{h+1}\neq0
\]
by Lemma~\ref{lem:transfinite-terms} and
Lemma~\ref{lem:kernel-powers}. Equality at
any stage makes the series stationary thereafter. Thus no stage below
$\omega^2$ is stationary, whereas
$I^{\omega^2}=I^{\omega^2+1}=0$.
\end{proof}

\section{Formal verification}\label{sec:formalization}

Theorems~\ref{thm:A}, \ref{thm:B}, \ref{thm:main},
and~\ref{thm:transfinite-length} have been formally verified in Lean~4.
Each of these statements in the introduction bears a clickable ``Lean 4''
badge at its upper right, linking to its formal statement and proof in
\texttt{Parafree/MainResults.lean}. The source code is publicly available
in the repository~\cite{IoninMikhailovLean2026}.

The development uses Lean~4.29.0, mathlib, and the LieRings library,
including its integral Poincar\'e--Birkhoff--Witt theorem. The toolchain
and dependencies are pinned. There are no \texttt{sorry} placeholders
or additional \texttt{axiom} declarations. Audits check all project
declarations and allow only Lean's standard axioms of propositional
extensionality, quotient soundness, and classical choice. The README
provides build instructions using \texttt{lake build} and commands
for these audits.

\end{document}